%% file: structural.tex
\documentclass{article}
\usepackage{graphicx} % Required for inserting images
\usepackage{todonotes}
\input{header}
\DeclareMathOperator{\di}{di}

\title{Structural Infinite-Exponent Partition\\Relations and Weak Choice Principles}
\author{Lyra A.\ Gardiner\footnote{Department of Pure Mathematics and Mathematical Statistics \& Trinity College, University of Cambridge\\\textsf{lag44@cam.ac.uk}}\ \ and Jonathan Schilhan\footnote{University of Vienna,
Institute of Mathematics,
Kurt Gödel Research Center,
Kolingasse 14-16,
1090 Vienna,
Austria\\\textsf{jonathan.schilhan@univie.ac.at}}}
\date{}

\begin{document}

\maketitle
We investigate infinite-exponent partition relations on arbitrary relational structures, with a focus on linear orders and graphs. Any such relation contradicts the Axiom of Choice. We show that there are some such relations which are consistent with \s{ZF} which imply the failure not just of Choice but also of the Kinna-Wagner Selection Principle \s{KWP}$_1$ and the Ordering Principle \s{O}.
\section{Introduction}\label{introduction:section}
\subsection{Background and motivation}\label{background:subsection}
The results in this paper are an offshoot of the work done in \cite{higheranalogues}. Several of our results in that paper show in \s{ZF} that for certain order types $\tau$, a partition relation with exponent $\tau$ can never hold on an order of the form $\twoalphalex$, for $\alpha$ an ordinal, and by extension can never hold on any linear order $\L$ which embeds into some $\twoalphalex$ (partition relations and the arrow notation used to express them will be explained in \S\ref{partitionrelations:subsection}). In \s{ZFC}, every linear order $\L$ embeds into some $\twoalphalex$, but this is not in general true in \s{ZF}. It is natural, then, to ask whether there consistently exists \emph{any} $\L$ for which
\[\L \rightarrow (\tau)^\tau.\]
The results of Section \ref{consistency:section} show that in a number of cases the answer is yes. The particular linear orders $\L$ we obtained had the property that the elements of $L$ were sets of sets of ordinals, leading naturally to the question of whether this was necessary, or whether there consistently existed such $\L$ with $L \subseteq \twoalpha$ for some $\alpha$. In Section \ref{choice:section} we show that it is in fact necessary, and that any such $\L$ has the property that the set $L$ does not inject into the power set of an ordinal, a failure of the Kinna-Wagner Selection Principle. Similar results in the setting of graphs are also proved in both sections.
\subsection{Partition relations}\label{partitionrelations:subsection}
For $A$, $B$ structures in some relational language, write $[A]^B$ to denote the set of substructures of $A$ which are isomorphic to $B$ (the set of \emph{copies of B in A}).\footnote{In structural Ramsey theory, this is usually denoted $\binom{A}{B}$, but we avoid this notation as it clashes with the notation used for polarised partition relations (cf.\ \cite[\S4]{ieprsonr}, \cite[\S4]{higheranalogues}).} For $A, B, C$ structures in the same language with $[A]^B$ and $[B]^C$ non-empty and $\chi$ a set, the \emph{partition relation}
\[A \rightarrow (B)^C_\chi\]
is the statement that for any $F : [A]^C \to \chi$, thought of as a \emph{colouring} of the copies of $C$ in $A$, there is some $H \in [A]^B$ which is \emph{homogeneous} (or \emph{monochromatic}) for $F$, in the sense that $\left|F \im [H]^C\right| = 1$. When $\chi$ is omitted it is understood to equal $2 = \{0,1\}$. The relation $A \rightarrow (B)^C_\chi$ is clearly preserved by replacing any of $A$, $B$, $C$ with isomorphic structures; in the particular setting of linear orders, we will often use \emph{order types}, defined in \S\ref{terminology:subsection} in place of concrete orders in this notation. The partition relation has monotonicity in its terms in the following sense: the relation $A \rightarrow (B)^C_\chi$ is preserved if $A$ is replaced by a \emph{larger} structure $A'$, if $B$ is replaced by a \emph{smaller} structure $B'$, or if $\chi$ is replaced by a smaller set $\chi'$.\footnote{More precisely, we mean $A'$, $B'$, $\chi'$ such that $[A']^A \neq \emptyset$, $[B]^{B'} \neq \emptyset$, and $\chi'$ injects into $\chi$.} We do not in general have monotonicity in the \emph{exponent} $C$.

Our work focuses on \emph{infinite-exponent partition relations} (IEPRs), and in particular on relations whose exponent $C$ contains proper subcopies of itself: in this setting, even the \emph{minimal relation} $A \rightarrow (C)^C$ is not guaranteed to hold. Indeed, under the Axiom of Choice, such a relation can never hold; it is in light of this that we work in \s{ZF} without Choice when discussing IEPRs.
\begin{prop}\s{(ZFC)}\label{trivialunderchoice:prop}
    Let $A$ and $C$ be structures such that $[A]^C \neq \emptyset$ and $[C]^C \supsetneq \{C\}$. Then
    \[A \centernot \rightarrow (C)^C.\]
\end{prop}
\begin{proof}
    Under \s{AC}, every set is well-orderable; fix an enumeration $\langle X_\alpha : \alpha < \gamma\rangle$ of $[A]^C$ for some $\gamma$. Say that $\alpha$ is the \emph{index} of $X_\alpha$. We inductively define a colouring $F : [A]^C \to 2$ with no homogeneous set.

    For $\alpha < \gamma$, at stage $\alpha$, if the value of $F(X_\alpha)$ has already been determined, proceed to stage $\alpha+1$; otherwise, determine the value of $F(X_\alpha)$ by means of the following procedure. If there exists a descending $\omega$-chain $X_\alpha = X_{\alpha_0} \supsetneq X_{\alpha_1} \supsetneq X_{\alpha_2} \supsetneq \dots$ such that $F(X_{\alpha_n})$ has not been decided for any $n \in \omega$, choose such a chain, and for each $n$ set
    \[F(X_{\alpha_n}) \coloneqq \begin{cases*}
        0 & if $n$ is even;\\
        1 & if $n$ is odd.
    \end{cases*}\]
    In this case say that the $X_{\alpha_n}$ have been \emph{coloured by alternation}. If it is not possible to find such a sequence, simply set $F(X_{\alpha}) = 0$.

    It is clear that any $X \in [A]^C$ which was coloured by alternation is not homogeneous for $F$, and also that if any $X_\alpha$ was coloured at a stage $\beta < \alpha$, it was coloured by alternation. We claim that no $X \in [A]^C$ is homogeneous for $F$; to see this, it suffices to show that for every $X \in [A]^C$, there is some $X' \in [X]^C$ which was coloured by alternation. Let $X = X_\alpha$ be arbitrary. Let $\alpha'$ be the minimal index of an element of $[X_\alpha]^C$, and consider $X_{\alpha'} \in [X_\alpha]^C$. Since $\alpha'$ is the minimal index of an element of $[X_\alpha]^C$, it is also the minimal index of an element of $[X_{\alpha'}]^C$, i.e.\ for any $X_\beta \in [X_{\alpha'}]^C$, $\beta \ge \alpha'$. If $X_{\alpha'}$ was coloured by alternation, we are done; otherwise, it must be the case that at stage $\alpha'$, some of the elements of $[X_{\alpha'}]^C$ already had colours assigned, such that it was impossible to find a descending chain of substructures whose colour under $F$ had not yet been determined. But for $F(X_\beta)$ to have been determined by stage $\alpha' < \beta$, it must have been coloured by alternation.
\end{proof}

It is therefore the case that any IEPR of the form $A \rightarrow (C)^C$ implies the failure of \s{AC}; in Section \ref{choice:section}, we show that certain IEPRs which are consistent with \s{ZF} also imply the failure of weaker choice principles than the full Axiom of Choice.
\subsection{Linear orders}\label{terminology:subsection}
Much of this paper relates to the combinatorics of linear orders. We reserve the letters $\sigma$ and $\tau$ for \emph{order types}, i.e.\ isomorphism classes of linear orders under order-isomorphism. Addition of linear orders or order types refers to concatenation, so e.g.\ $\sigma + \tau$ is the order type of a linear order consisting of an initial segment ordered as $\sigma$ followed by a final segment ordered as $\tau$. Multiplication is colexicographic, so $\sigma\tau$ refers to the order type of $\tau$-many copies of $\sigma$. For $\sigma$ an order type, $\sigma^*$ denotes its reverse. $\omega$ is the order type of the natural numbers, and $\eta$ is the order type of the rational numbers, each under their usual ordering. The \emph{embeddability relation} $\le$ is defined as follows: $\sigma \le \tau$ iff any $\L$ of type $\tau$ has a suborder of type $\sigma$.

For $\L$ a linear ordering, a \emph{condensation class} of $\L$ (under the finite condensation) is an equivalence class of the relation $x \sim_L y$ given by
\[x \sim _L y \iff \{z \in L : x \le z \le y \text{ or } y\le z \le x\}\text{ is finite.}\]
Condensation classes are necessarily either finite or they are ordered as one of $\omega$, $\omega^*$, or $\omega^* + \omega$. We will be particularly interested in the condensation classes of certain subsets $A \subseteq L$ of an ambient linear order $\L$, equipped with the induced suborder. For $A \subseteq L$, write $\ccomega(A)$ for the set of condensation classes of $A$ ordered as $\omega$, and $\ccomegastar(A)$ for the set of condensation classes of $A$ ordered as $\omega^*$.

A linear order $\L$ is said to be \emph{scattered} if no suborder of $\L$ is densely ordered; equivalently, for $L$ well-orderable, $\L$ is scattered iff $[\L]^\eta = \emptyset$. Iterations of the finite condensation allow for a detailed analysis of the structure of scattered orders; see \cite[\S5]{rosenstein}.

Intervals are written with parentheses and brackets as usual; $(x,\rightarrow)$ and $[x,\rightarrow)$ denote final segments, and $(\leftarrow,x)$, $(\leftarrow, x]$ denote initial segments. Some of our proofs involve sets equipped with multiple different linear orderings; in these settings, we write e.g.\ $(x,y)_<$ and $(x,y)_\prec$ to disambiguate which order is meant when speaking of intervals.

An order type $\sigma$ is \emph{(additively) indecomposable} if whenever $\sigma$ is written as a sum $\sigma = \tau + \varphi$, either $\sigma \le \tau$ or $\sigma \le \varphi$. Further, $\sigma$ is \emph{strictly indecomposable to the left (resp.\ right)} if whenever $\sigma$ is written as a sum $\sigma = \tau + \varphi$, $\sigma \le \tau$ (resp.\ $\sigma \le \varphi$).

The orders $\twoalphalex$ for $\alpha$ an infinite ordinal are of particular significance to our work; the paper \cite{higheranalogues} was an investigation into the partition relations which hold on these orders. They are ordered by the \emph{lexicographic order} $\lex$, where for $x, y \in \twoalpha$,
\[x \lex y \iff x(\delta) < y(\delta),\]
where $\delta = \delta(x,y)$ is the minimal ordinal $< \alpha$ on which the sequences $x$ and $y$ disagree.

The set $\lessomega$ of finite binary sequences is also very important to our work; we equip it with a tree ordering and a linear ordering. The tree $\lessomegatree$ is ordered by $\sqsubseteq$, where $x \sqsubseteq y$ iff $x$ is an initial segment of $y$, i.e.\ $y \restriction |x| = x$. The linear order $\lessomegalex$ is obtained by extending the definition of the lexicographic ordering so that $x \lex y$ includes the cases $x ^\frown\langle 1\rangle \sqsubseteq y$ and $x \sqsupseteq y^\frown \langle0\rangle$; this makes $\lessomegalex$ a linear order ordered as $\eta$.

\section{Consistency results}\label{consistency:section}
We begin with the key motivating result which led to the work done in this paper:
\begin{fact}\label{motivation:fact}
    Let $\tau$ be a countable linear order type such that either $\omega\omega^* \le \tau$ or $\omega^*\omega \le \tau$. Then for any ordinal $\alpha$,
    \[\twoalphalex \centernot \rightarrow (\tau)^\tau.\]
\end{fact}
This is a consequence of Theorems 1 and 2 of \cite{higheranalogues}, corresponding to the case in which $\tau$ is non-scattered and the case in which $\tau$ is scattered, respectively:
\begin{thm*}
    \textnormal{(Theorem 1, \cite{higheranalogues})} Let $\tau \ne 0$ be an order type with $\tau + \tau \le \tau$. Then for any ordinal $\alpha$, \[\twoalphalex\centernot \rightarrow (\tau)^\tau.\]
\end{thm*}
\begin{thm*}
    \textnormal{(Theorem 2, \cite{higheranalogues})} Let $\alpha$ be an ordinal and $\tau$ a scattered, well-orderable order type. If $\omega\omega^* \le \tau$ or $\omega^*\omega \le \tau$ then
    \[\twoalphalex \centernot \rightarrow (\tau)^\tau.\]
\end{thm*}
We show now that Fact \ref{motivation:fact} does not generalise to all linear orders:
\begin{thm}\label{consistency:thm}
    It is consistent with \s{ZF} that there exists a linear order $\L$ such that for any countable $\tau$,
    \[\L \rightarrow (\tau)^\tau.\]
\end{thm}
We first prove the consistency of this statement with \s{ZFA} by means of permutation models, and then show that the result lifts to a \s{ZF} result using symmetric systems.\footnote{\s{ZF(C)A} is the theory of \emph{Zermelo-Fraenkel set theory with atoms}; see \cite[\S4]{jechaxiomofchoice} for a detailed treatment of this theory and of the theory of permutation models.}
\begin{lemma}\label{zfa:lemma}
    It is consistent with \s{ZFA} that there exists a linear order $\L$ such that for any countable $\tau$,
    \[\L \rightarrow (\tau)^\tau.\]
\end{lemma}
\begin{proof}
    Let $M$ be a model of \s{ZFCA + CH} whose set of atoms $A$ has cardinality $\aleph_1$, and let $\prec\,\, \subseteq A \times\! A$ be such that $\Aprec$ is a saturated linear order (this is possible as $M \models 2^{\aleph_0} = \aleph_1$: see \cite[Corollary 4.3.13]{marker}). Let $\mathscr G \le \textnormal{Aut}(A)$ be the group of all permutations of $A$ which preserve the order $\prec$, i.e.\ the group of all order-automorphisms of $\Aprec$, and let $\mathscr F \subseteq \mathcal P(\mathscr G)$ be the normal filter generated by subgroups of $\mathscr G$ which fix countable subsets of $A$. Write $N \subseteq M$ for the permutation model generated by $\mathscr G$ and $\mathscr F$. It is clear that $\Aprec \in N$; observe that all countable subsets of $A$ in $M$ are also elements of $N$, as for $E \subseteq A$ countable, $\fix(E) \le \sym(E)$, so $\sym(E) \in \mathscr F$. We claim that in $N$, the order $\Aprec$ has the desired property.
    
    Let $\tau$ be a countable order type and let $F : [\Aprec]^\tau \to 2$ be a colouring in $N$. Since $F$ is symmetric, it has some support $E \subseteq A$, a countable set of atoms. Then by the $\omega_1$-saturation of $\Aprec$ in $M$, we can find (in $M$) an interval $I$ of $\Aprec$ such that $I \cap E = \emptyset$. Then let $H \in [I]^\tau$ be arbitrary; by the above, $H \in N$, and $M$ and $N$ agree on $[H]^\tau$. We claim that $H$ is homogeneous for $F$. Let $H' \in [H]^\tau$, and let $i : E \cup H \to E \cup H'$ be a partial automorphism of $A$ which fixes every element of $E$ and which maps $H$ to $H'$ order-preservingly, noting that such an $i$ exists as the convex hull of $H$ is disjoint from $E$. Since $\Aprec$ is saturated, $i$ can be extended to a full order-automorphism, $\pi$, say; treating $\pi$ as an automorphism of $N$, we have $\pi(F)(\pi(H)) = F(H)$, but $\pi(F) = F$ as $\pi \in \fix(E)$ and $\pi(H) = H'$ by construction, so $F(H') = F(H)$. Since $H' \in [H]^\tau$ was arbitrary, it follows that $H$ is homogeneous for $F$.
\end{proof}
\begin{proof}[Proof of Theorem \ref{consistency:thm}]
    For $\L$ a linear order, $\tau$ an order type, the statement
    \[\L \rightarrow (\tau)^\tau\]
    is determined by $\mathcal P^3(L)$, as $<\,\, \in \mathcal P^3(L)$, any $H \subseteq L$ is of course an element of $\mathcal P(L)$, and any colouring $F : [\L]^\tau \to 2$ can be encoded as a set of subsets of $L$, i.e.\ an element of $\mathcal P^2(L)$.
    
    Now, let $M$, $\Aprec$, and $N$ be as in the proof of Lemma \ref{zfa:lemma}, and let $V \subseteq M$ be the kernel of $M$, i.e.\ the class of hereditary sets, so $V \models \s{ZFC}$. Then by the Jech-Sochor theorem (cf.\ \cite[Theorem 6.1]{jechaxiomofchoice}), there is a symmetric extension $V(G) \models \s{ZF}$ of $V$ and a set $\bar{A} \in V(G)$ such that
    \[\left(\mathcal P^3(A)\right)^N\text{ is }\in\!\text{-isomorphic to }\left(\mathcal P^3(\bar{A})\right)^{V(G)}.\]
    It follows that in $V(G)$, there is an ordering $\prec' \,\,\subseteq\bar{A}\times\!\bar{A}$ such that for all countable $\tau$
    \[\langle \bar{A},\prec'\rangle \rightarrow (\tau)^\tau.\qedhere\]
\end{proof}
The Jech-Sochor theorem uses sets of sets of ordinals in a generic extension to mimic the role of atoms in \s{ZFA}, so the order $\langle \bar{A},\prec'\rangle$ built above consists of sets of sets of ordinals. The question of whether this is necessary or whether one can consistently find an $\L$ with the same property such that $L$ is a set of ordinals inspired the work of the next section.

We remark here that an almost identical proof to the above yields the following result also:\footnote{We modify the proof of Lemma \ref{zfa:lemma} in the following way. Start with  a model of \s{ZFCA + CH} whose set of atoms $A$ can be equipped with an edge relation making it a saturated graph; where in that proof we found an interval $I$ disjoint from the countable support set $E$, in this proof we appeal to $\omega_1$-saturation to find a copy of $R$ disjoint from $E$.}
\begin{prop}\label{graphconsistency:prop}
    It is consistent with \s{ZF} that there exists a graph $G$ such that for any countable graph $\Gamma$,
    \[G \rightarrow (\Gamma)^\Gamma.\]
\end{prop}
\section{The failure of weak choice principles}\label{choice:section}
The particular fragments of choice with which we concern ourselves in this paper are \s{KWP}$_1$ and \s{O}, which we introduce below:
\begin{defn}
    The \emph{Kinna-Wagner Selection Principle} \s{KWP}$_1$ is the statement that every set injects into the power set of an ordinal,
    \[\forall A \,\,\exists \alpha \in \text{Ord} \,\,\exists f : A \to \mathcal P(\alpha)\text{ injective.}\]
\end{defn}
Some comments: this principle was introduced in \cite{kinnawagner}, originally as a different statement which was shown equivalent to the formulation given above. More generally one can study the principle \s{KWP}$_\beta$ for $\beta$ an ordinal, which is the statement that $\forall A \,\,\exists \alpha \in \text{Ord} \,\,\exists f : A \to \mathcal P^\beta(\alpha)\text{ injective,}$ so the Axiom of Choice is \s{KWP}$_0$. The statement \s{KWP} then denotes the statement that \s{KWP}$_\beta$ holds for some $\beta$; see \cite{kwpasafjonathan}. In this paper we concern ourselves only with \s{KWP}$_1$.

\begin{defn}
    The \emph{Ordering Principle} \s{O} is the statement that every set can be linearly ordered, i.e.\ for any $A$ there is a relation $<\,\,\, \subseteq A \times A$ which is transitive, irreflexive, and trichotomous on $A$.
\end{defn}

\begin{obs}
    \[\s{KWP}_1 \implies \s{O}.\]
\end{obs}

\begin{proof}
    Assume \s{KWP}$_1$, and let $A$ be a set. Then there is an ordinal $\alpha$ and an injection $i : A \hookrightarrow \twoalpha$. It follows that the relation $<$ defined on $A$ by, for $a, b \in A$,
    \[a < b \iff i(a)\lex i(b)\]
    is a linear order on $A$.
\end{proof}

The reverse implication does not hold; see \cite{felgnerordering} or \cite{pincusfm}.

We now return to the structural IEPRs shown consistent in \S\ref{consistency:section} and show that the existence of a structure satisfying such a partition relation yields a failure of one of these choice principles. For the purposes of the proof it will be easier to phrase our results contrapositively, i.e.\ in the form ``if the structure $A$ does not witness the failure of some choice principle, then $A \centernot \rightarrow (C)^C$". In order to build our colourings with no homogeneous sets, we make use of the following concept from \cite{higheranalogues}:

\begin{defn}
    Let $A$, $C$ be structures with $[A]^C \neq \emptyset$. A set $\D \subseteq [A]^C$ is \emph{dense} if for any $X \in [A]^C$, there is some $X' \in [X]^C$ with $X' \in \D$.
\end{defn}
A colouring of some dense $\D$ with no homogeneous set extends to a colouring of the entirety of $[A]^C$ with no homogeneous set, e.g.\ by just sending every element of $[A]^C\setminus \D$ to $0$.
\begin{obs}
    Let $A$, $C$, be structures, let $\D \subseteq [A]^C \neq \emptyset$ be dense, and let $F : \D \to 2$ have no homogeneous set, in the sense that for any $X \in \D$, there is some $X' \in [X]^C \cap \D$ with $F(X) \neq F(X')$. Then $F$ extends to a colouring $F' : [A]^C \to 2$ witnessing $A \centernot \rightarrow (C)^C$. \qed
\end{obs}

\begin{thm}\label{countablenonscattered:thm}
    Let $\tau$ be a countable non-scattered order type, and let $\L$ be a linear order. If the set $L$ injects into the power set of an ordinal, then
    \[\L \centernot \rightarrow (\tau)^\tau.\]
\end{thm}
Equivalently, if the linear order $\L$ witnesses $\L \rightarrow (\tau)^\tau$ for $\tau$ some countable non-scattered order type, then the set $L$ witnesses the failure of \s{KWP}$_1$. We first prove a lemma of independent interest:
\begin{lemma}\label{canonisationonq:lemma}
    Let $\prec$ be an arbitrary linear order on $\Q$, and let $<$ denote the usual linear order on $\Q$. Then there exists some $A \in [\Qq]^\eta$ satisfying one of the following conditions:
    \begin{enumerate}
        \item $\Aprec \cong \omega$;
        \item $\Aprec \cong \omega^*$;
        \item $\prec$ is identical to $<$ on $A$;
        \item $\prec$ is the reverse of $<$ on $A$.
    \end{enumerate}
\end{lemma}
\begin{proof}[Proof of Lemma \ref{canonisationonq:lemma}]
    Fix such an ordering $\prec$. For the duration of this proof we fix $\langle q_n :n \in \omega \rangle$ an enumeration of $\Q$; when we refer to a rational ``of minimal index", we mean $q_n$ with $n$ minimal according to this enumeration.

    First observe that whenever $\L$ is ordered as $\eta$ and $A \subseteq L$ is scattered under $<$, $\langle L \setminus A, < \rangle$ is still ordered as $\eta$; otherwise, it would have to be the case that $A$ contained an interval of $\L$ (either between two elements of $L$ or an initial or final segment); but then $A$ is not scattered, a contradiction.
    \begin{claim}\label{reducetoomega:claim}
        If $\prec$ is such that whenever $\Iprec$ is a proper initial segment of $\Qprec$, $\Ii$ is scattered, then there is $A \in [\Qq]^\eta$ with $\Aprec \cong \omega$.
    \end{claim}
    \begin{claimproof}
        Note that the given condition implies that $\Qprec$ has no maximal element and thus has cofinality $\omega$. Fix a cofinal sequence $x_0 \prec x_1 \prec \dots$ in $\Qprec$, and for each $n \in \omega$ set $F_n \coloneqq \{y \in \Q : x_n \preceq y\}$. Then each $F_n$ is dense in $\Qq$ (as $\Q \setminus F_n$ is scattered under $<$). We can now run a back-and-forth construction to build $A \in [\Qq]^\eta$ with the desired property.
        
        We build $\langle a_n : n \in \omega\rangle$ recursively like so: take $a_0 \in F_0$, and, given $a_0, \dots, a_{n-1}$, take $a_n \in F_n$ such that for all $m < n$,
        \begin{align*}
            a_n < a_m \iff q_n < q_m&\text{, and}\\
            a_n > a_m \iff q_n > q_m&\text{;}
        \end{align*}
        then set $A \coloneqq \{a_n : n \in \omega\}$.\footnote{Observe that this does not use Choice, as we can at each step take $a_n$ to be minimal according to the enumeration $\langle q_n :n \in \omega \rangle$.} Clearly $\Aa \cong \Qq$, so $A \in [\Qq]^\eta$. Now, since for each $n \in \omega$, $A$ contains only finitely many elements of $F_n \setminus F_{n+1}$, $\Aprec$ is ordered as $\omega$.
    \end{claimproof}
        By symmetry, the corresponding statement with initial segments replaced by final segments and the order type $\omega$ replaced by $\omega^*$ also holds.
    \begin{claim}\label{isfsnonscattered:claim}
        Suppose $A \in [\Qq]^\eta$ has the property that for all $A' \in [\Aa]^\eta$, $\langle A', \prec \rangle \not \cong \omega$ and $\langle A' \prec \rangle \not \cong \omega^*$. Then there exists $B \in [\Aa]^\eta$ with the following property:

        For any $I \subseteq B$ such that $\Iprec$ is either an initial or final segment of $\Bprec$, $\Ii$ is non-scattered.

        Moreover, this $B$ is determined from $A$ in a canonical way.
    \end{claim}
    \begin{claimproof}
        Suppose some (necessarily proper) initial segment $I$ of $\Aprec$ has that $\Ii$ is scattered. Taking a union of all such $\prec$-initial segments, we obtain $J \subseteq A$, the maximal initial segment of $\Aprec$ with the property that for any proper initial segment $I$ of $\Jprec$, $\Ii$ is scattered.

        If $\Jj$ is non-scattered, i.e.\ $\eta \le \otp \Jj$, then letting $T \in [\Jj]^\eta$ and applying Claim \ref{reducetoomega:claim} with $\langle T, < \rangle$ in place of $\Qq$, we obtain $A' \in [\langle T, < \rangle]^\eta$ with $\langle A', \prec \rangle \cong \omega$, contradicting our assumption that no $A' \in [\Aa]^\eta$ has this property.

        So $\Jj$ is scattered. It follows that $B' \coloneqq A \setminus J$ is ordered as $\eta$ and has the property that for any $I$ an initial segment of $\langle B', \prec \rangle$, $\Ii$ is non-scattered.

        Running the symmetric procedure on $B'$, we reduce to $B \in [\langle B', <\rangle]^\eta \subseteq [\Aa]^\eta$ with the property that whenever $I$ is an initial segment \emph{or} a final segment of $\Bprec$, $\Ii$ is non-scattered.

        Note finally that $B$ was obtained from $A$ by removing uniquely-determined initial and final segments of $\Aprec$, and so this procedure canonically determines $B$.
    \end{claimproof}
    If there exists some $A \in [\Qq]^\eta$ with $\Aprec \cong \omega$ or $\Aprec \cong \omega^*$ we are done, so assume that no such $A$ exists. We now define two functions
    \begin{align*}
        q &: \lessomega \to \Q,\\
        Q &: \lessomega \to [\Qq]^\eta
    \end{align*}
    recursively; our desired $A$ will be a particular subset of $q \im \lessomega$.

    Set $q(\emptyset)$ to be any element of $\Q$ (e.g.\ $q_0$ in the enumeration of $\Q$ fixed above), and $Q(\emptyset) = \Q$. For $s \in \lessomega$, given $q(s)$ and $Q(s)$, we find $Q(\szero)$ and $Q(\sone)$ in the following way:

    First, apply Claim \ref{isfsnonscattered:claim} to $Q(s) \cap (-\infty, q(s))_<$ and $Q(s) \cap (q(s), \infty)_<$ to obtain $B_0$ and $B_1$, respectively, each with the property that any $\prec$-initial segment or $\prec$-final segment is non-scattered in the order $<$. Let $n \in \omega$ be minimal such that either
    \begin{enumerate}[(a)]
        \item $B_0' \coloneqq B_0 \cap (\leftarrow,q_n)_\prec$ and $B_1' \coloneqq B_1 \cap (q_n,\rightarrow)_\prec$ are both non-empty, or
        \item $B_0''\coloneqq B_0 \cap (q_n,\rightarrow)_\prec$ and $B_1'' \coloneqq B_1 \cap (\leftarrow,q_n)_\prec$ are both non-empty.
    \end{enumerate}
    In case (a), $\langle B_0',<\rangle$ and $\langle B_1',<\rangle$ are both non-scattered, and as such each consist of a densely-ordered sum of scattered intervals; by taking the minimal-indexed element of each of these scattered intervals, then removing the minimal or maximal element if necessary, we can find $A_0 \in [\langle B_0',<\rangle]^\eta$, $A_1 \in [\langle B_1',< \rangle]^\eta$ in a canonical way; by construction, $A_0 \prec A_1$. In case (b), we instead obtain $A_0 \in [\langle B_0'',< \rangle]^\eta$, $A_1 \in [\langle B_1'', <\rangle]^\eta$ with $A_0 \succ A_1$.\footnote{It is possible that both (a) and (b) hold; in this case we proceed as in case (a).} Set $Q(\szero) = A_0$ and $Q(\sone) = A_1$, and let $q(\szero) \in Q(\szero)$, $q(\sone) \in Q(\sone)$ of minimal index.

    We now find our desired $A$ as a subset of $q \im \lessomega$. Consider the colouring $c: \lessomega \to 2$ given by, for $s \in \lessomega$,
    \[c(s) = \begin{cases*}
        0 & if $Q(\szero) \prec Q(\sone)$;\\
        1 & if $Q(\szero) \succ Q(\sone)$.
    \end{cases*}\]
    By, e.g.\ the Halpern–Läuchli theorem in the setting of one tree, there exists $H \subseteq \lessomega$ such that $\langle H, \sqsubseteq\rangle$ is isomorphic to $\langle \lessomega,\sqsubseteq\rangle$ on which $c$ is homogeneous; then $A \coloneqq q \im H$ has the desired property: if $H$ is homogeneous with colour 0, then for any $s, t \in \lessomega$ with $s \lex t$, we have that $q(s) \prec q(t)$, so $A$ satisfies condition 3; if instead $H$ is homogeneous with colour 1, $s \lex t$ implies $q(s) \succ q(t)$, and so $A$ satisfies condition 4.
\end{proof}
\begin{proof}[Proof of Theorem \ref{countablenonscattered:thm}]
\setcounter{scratch}{\value{thmcount}}
\setcounter{thmcount}{\getrefnumber{countablenonscattered:thm}}
    Fix $\tau$, $\L$, and an injection $i : L \hookrightarrow \twoalpha$ for some ordinal $\alpha$. We will build  a colouring with no homogeneous set on a dense subset of $[\L]^\tau$.

    For $j \in \{0,1\}$, define $\D_j \subseteq [\L]^\tau$ like so:
    \begin{align*}
        \D_0& \coloneqq \{A \in [\L]^\tau : i \im A \text{ is ordered as }\omega\text{ or }\omega^*\text{ under }\lex\};\\
        \D_1& \coloneqq \{A \in [\L]^\tau : i \text{ is order-preserving or order-reversing on }A\},
     \end{align*}
    where here ``order-preserving" (resp. ``reversing") means with respect to the orders $\L$ and $\twoalphalex$, and define $\D \coloneqq \D_0 \cup \D_1$. Observe that each of $\D_0$, $\D_1$ are downwards closed.
    \setcounter{claimcount}{0}
    \begin{claim}
        $\D$ is dense in $[\L]^\tau$.
    \end{claim}
    \begin{claimproof}
        Let $A \in [\L]^\tau$; we will show that there is some $A' \in [A]^\tau$ with $A' \in \D$. Note that $A$ can be reduced to some $Q \in [A]^\eta$; then applying Lemma \ref{canonisationonq:lemma} to $Q$ with the orders $<$ and $\prec \,\coloneqq i^{-1}(\lex)$, we reduce further to some $Q' \in [\langle Q,<\rangle]^\eta$ such that either $i \im Q'$ is ordered as one of $\omega$ or $\omega^*$, or $i$ is order-preserving or order-reversing on $Q'$. Now let $A'$ be any element of $[\langle Q',<\rangle]^\tau$ (these exist as $\langle Q',<\rangle$ is ordered as $\eta$); it is clear that $A'$ retains whichever of the above properties $Q'$ had under $i$.
    \end{claimproof}
    Now define $c : \D \to 2$ like so: for $A \in \D_0$, write $\{a_n : n \in \omega\}$ for the natural enumeration of the elements of $A$ determined by the ordering of $i \im A$ in type $\omega$ or $\omega^*$; then set
    \[c(A) = \begin{cases*}
        0 & if $a_0 > a_1$;\\
        1 & if $a_0 < a_1$.
    \end{cases*}\]
    To see that no $A \in \D_0$ can be homogeneous for this, note that since $\tau$ is non-scattered, we can find $n < m$ such that $a_n < a_m$ and each of $I_0 := (\leftarrow, a_n)_< \cap A$, $I_1 := (a_n, a_m)_< \cap A$ and $I_2 := (a_m, \rightarrow)_< \cap A$ are non-scattered. Letting $B := \{ a_k : k<m, k \neq n \}$, which is finite, for each $l<3$, we also have that $I_l \setminus B$ is non-scattered, which further contains a copy $J_l$ of $I_l$. It follows that $A' := J_0 \cup J_1 \cup J_2 \cup \{a_n, a_m\}$ is isomorphic to $A$ and $c(A') = 0$. Similarly, we can achieve that $c(A') = 1$.
     
    %To see that no $A \in \D_0$ can be homogeneous for this, note that since $\tau$ is non-scattered, it can be decomposed into a sum of maximal scattered intervals, ordered densely; since $\tau$ is countable these intervals are ordered as one of $\eta$, $\eta + 1$, $1 + \eta$, or $1 + \eta + 1$. Let us call the pieces of this decomposition the ``scattered parts" of $\tau$; then given any $x, y \in A$ from distinct non-extremal scattered parts of $A$, there is some $A' \in [A]^\tau$ with $\{x,y\} = \{a'_0, a'_1\}$; for appropriate choices of $x,y$ we can therefore find $A' \in [A]^\tau$ taking either value under $c$.

    For $A \in \D_1$, we simply apply the colouring from \cite[Theorem 1]{higheranalogues} to $i \im A$; as this colouring witnesses that $\twoalphalex \centernot \rightarrow (\tau)^\tau$ (or that $\twoalphalex \centernot \rightarrow (\tau^*)^{\tau^*}$) it is clear that there is no $A \in \D_1$ homogeneous for this.
\end{proof}\setcounter{thmcount}{\value{scratch}}
\begin{thm}\label{omegaomegastarkw:thm}
    Let $\tau$ be a well-orderable scattered order type with $\omega\omega^* \le \tau$ or $\omega^*\omega \le \tau$, and let $\L$ be a linear order. If the set $L$ injects into the power set of an ordinal, then
    \[\L \centernot \rightarrow (\tau)^\tau.\]
\end{thm}
We will first prove the special case of this theorem where $\tau$ is additively indecomposable, then complete the proof of Theorem \ref{omegaomegastarkw:thm} by explaining how to adapt this proof to the general case.
\begin{lemma} \label{indecscatteredkw:lemma}
 Let $\tau$ be a well-orderable additively indecomposable scattered order type with either $\omega\omega^* \le \tau$ or $\omega^* \omega \le \tau$, and let $\L$ be a linear order. If the set $L$ injects into the power set of an ordinal, then
 \[\L \centernot \rightarrow (\tau)^\tau.\]
\end{lemma}
We first introduce some definitions and notation used throughout the proofs of Lemma \ref{indecscatteredkw:lemma} and Theorem \ref{omegaomegastarkw:thm}. Particular sets ordered as $\omega$ or $\omega^*$ will be of key importance throughout; we adopt the convention that if $X \in [\L]^\omega$ for some $\L$, $\langle x_n : n \in \omega\rangle$ is the natural increasing enumeration of $X$, so $X = \{x_n : n \in \omega\}$ and $x_0 < x_1 < \dots$. Similarly, if $X \in [\L]^{\omega^*}$, $\langle x_n : n \in \omega\rangle$ is the natural decreasing enumeration of $X$. We now introduce a useful notion from \cite{higheranalogues}:

For $X \in [\twoalphalex]^\omega$, say that $X$ is \emph{canonised} if for $x, y \in X$ with $x < y$, the value of $\delta(x,y)$ is determined by $x$; symmetrically, for $X \in [\twoalphalex]^{\omega^*}$, say that $X$ is canonised if the value of $\delta(x,y)$ for $x < y \in Y$ is determined by $y$. In either case, this is equivalent to the condition that $\delta(x_0,x_1) < \delta(x_1,x_2) < \delta(x_2,x_3) < \dots$.
\begin{obs}\label{canonisation:obs}
    Given any $X \in [\twoalphalex]^\omega \cup [\twoalphalex]^{\omega^*}$, there is $X' \in [X]^{\otp(X)}$ which is canonised.\footnote{In fact, one can define a uniform procedure for picking out such an $X'$ canonically: see \cite[Lemma 9]{higheranalogues}.}\qed
\end{obs}
The proofs start by assuming the existence of some injection $i : L \hookrightarrow \twoalpha$. Given a fixed such $i$, for $O \in [\L]^\omega$ or $O \in [\L]^{\omega^*}$, say $O$ is \emph{$i$-canonised} if the following two conditions hold:
\begin{enumerate}
    \item $i$ is order-preserving or order-reversing on $O$;
    \item $i \im O$ is canonised in $\twoalphalex$.
\end{enumerate}

Throughout both proofs the splitting levels (in $\twoalphalex$) between elements of $i \im A$ for $A \in [\L]^\tau$ will be significant, so we introduce the following notation to avoid clutter: for $x, y \in L$, write
\[\delta_i(x,y) \coloneqq \delta(i(x),i(y)).\]
In particular, it will often be important for us to consider two such consecutive quantities in copies of $\omega$ or $\omega^*$: for $X \in [\L]^\omega \cup [\L]^{\omega^*}$ and $n \in \omega$, write

\[d_n(X) \coloneqq (\delta_i(x_n,x_{n+1}),\delta_i(x_{n+1},x_{n+2})).\]
For $n \in \omega$ fixed, the $d_n(X)$ are ordered according to the lexicographic ordering on $\alpha \times \alpha$. The quantity $d_2(X)$ will be important in one case in the proof of Lemma \ref{indecscatteredkw:lemma}, and the quantity $d_0(X)$ will be important in the proof of Theorem \ref{omegaomegastarkw:thm}.
\begin{proof}[Proof of Lemma \ref{indecscatteredkw:lemma}]
    Fix $\tau$, $\L$, and an injection $i : L \hookrightarrow \twoalpha$ for some ordinal $\alpha$. We appeal here to a result of Jullien:
    \begin{subfact}\label{jullien:fact}
    \textnormal{(\!\!\cite{jullienthesis})}
        Any indecomposable scattered order is either strictly indecomposable to the right or to the left.\footnote{This result was proved in \s{ZFC} but still holds in \s{ZF} for well-orderable scattered orders; see \cite[Footote 6]{higheranalogues}.}
    \end{subfact}
    Wlog $\tau$ is strictly indecomposable to the right. We will define a colouring $c: \D \to 2$ on a dense subset $\D$ of $[\L]^\tau$ with no homogeneous set by means of the following strategy: $\D$ will have the property that every $A \in \D$ has an interval which can be picked out and which is followed by (at least) two consecutive elements of $A$. Call this the \emph{distinguished interval} of $A$; roughly speaking, we first colour according to whether or not $i$ is order-preserving on the
    two elements of $A$ immediately following the distinguished interval. Since $\tau$ is strictly indecomposable to the right, for any $x < y \in A$ both above the distinguished interval, we can reduce to some $A' \in [A]^\tau$ such that $x$ and $y$ are the two elements of $A'$ picked out in this way; it follows that for $A$ to be homogeneous for this colouring, it must be the case that $i$ is either order-preserving or order-reversing on the final segment of $A$ above the distinguished interval. We then colour those $A$ with this property by applying the colouring $C$ from \cite[Theorem 2]{higheranalogues} to the image under $i$ of this final segment of $A$; since $i$ is either order-preserving or order-reversing on this final segment, its image in $\twoalphalex$ is ordered as some order type $\varphi$ with the property that $C$ witnesses $\twoalphalex \centernot \rightarrow (\varphi)^\varphi$. It therefore follows that none of these sets will be homogeneous for $c$ either. 

    \begin{case} $\tau$ has a condensation class ordered as $\omega^*$ or a condensation class ordered as $\zeta$.\end{case}

    Here our distinguished interval (and the dense subset of $[\L]^\tau$ on which we can guarantee that such an interval exists) has a somewhat complex definition. Let us write
    \[[A]^{\omega^*}_c \coloneqq \{X \in [A]^{\omega^*}: X \text{ is convex in }A\}.\]
    Say that $X \in [A]^{\omega^*}_c$ is \emph{almost $i$-canonised} if $X \setminus \{x_0,x_1\}$ is $i$-canonised, and associate each such $X$ with the quantity $d_2(X) = (\delta_i(x_2,x_3),\delta_i(x_3,x_4))$.
    \begin{claim}\label{omegastarcase:claim}
        For any $A \in [\L]^\tau$, there is $A' \in [A]^\tau$ with the following two properties:
        \begin{enumerate}
            \item There is at least one $X \in [A']^{\omega^*}_c$ which is almost $i$-canonised;
            \item Amongst all almost $i$-canonised $X \in [A]^{\omega^*}_c$ on which $d_2(X)$ is minimal, there is a \emph{leftmost} one (i.e.\ one which is $<$-minimal).
        \end{enumerate}
    \end{claim}
    \begin{claimproof}
        First observe that for any $A \in [\L]^\tau$, there is $A^* \in [A]^\tau$ with at least one $X \in [A^*]^{\omega^*}_c$ which is almost $i$-canonised: let $Y = \{y_n : n \in \omega\} \in [A]^{\omega^*}_c$ be arbitrary. Then $Y \setminus \{y_0, y_1\}$ is also ordered as $\omega^*$ in $\L$, and as such there is some $Y' \in [Y]^{\omega^*}$ on which $i$ is order-preserving or order-reversing, e.g.\ by Ramsey's theorem. Reduce further to some $Y'' \in [Y']^{\omega^*}$ such that $i \im Y''$ is canonised. Set
        \[A^* \coloneqq (A \setminus Y)\cup Y'' \cup \{y_0,y_1\}.\]
        Then $Y'' \cup \{y_0,y_1\} \in [A^*]^{\omega^*}_c$ is almost $i$-canonised. Write \[\mathcal A \coloneqq\{B \in [A]^\tau : \text{at least one }X \in [B]^{\omega^*}_c\text{ is almost }i\text{-canonised}\};\] it follows from the above that $\mathcal A$ is non-empty. For $B \in \mathcal A$, write
        \[\dmin_2(B) \coloneqq \min \{d_2(X) : X \in [B]^{\omega^*}_c\text{ almost }i\text{-canonised}\},\]
        and let $A^\dagger$ be an element of $\mathcal A$ which minimises $\dmin_2(A^\dagger)$. Then there is at least one almost $i$-canonised $X \in [A^\dagger]^{\omega^*}_c$ such that $d_2(X) = \dmin_2(A^\dagger)$ is minimal, but this is not necessarily unique. We further reduce to $A' \in [A^\dagger]^\tau$ with $\dmin_2(A') = \dmin_2(A^\dagger)$, with a leftmost witness $X$, in the following way: pick some almost $i$-canonised $X \in [A^\dagger]^{\omega^*}_c$ with $d_2(X) = \dmin_2(A^\dagger)$, and reduce to $A' \in [A^\dagger]^\tau$ by, for all almost $i$-canonised $Z \in [A^\dagger]^{\omega^*}_c$ with $d_2(Z) = \dmin_2(A^\dagger)$ and $Z < X$, removing $z_3$. Observe that this is well-defined, as for any $Z_0, Z_1 \in [A^\dagger]^{\omega^*}_c$, if $d_2(Z_0) = d_2(Z_1)$ then either $Z_0 = Z_1$ or $Z_0 \cap Z_1 = \emptyset$, and $A'$ still has order type $\tau$ as we have simply replaced some of its intervals ordered as $\omega^*$ by proper subsets also ordered as $\omega^*$.
        
        Now $X \in [A']^{\omega^*}_c$ still has $d_2(X) = \dmin_2(A') = \dmin_2(A^\dagger)$, which is minimal, but for any other almost $i$-canonised $Z \in [A^\dagger]^{\omega^*}_c$ with $d_2(Z) = d_2(X)$ and $Z < X$, $Z$ has been replaced by $Z' \coloneqq Z \setminus \{z_3\}$, and $d_2(Z') = (\delta_i(z_2,z_4),\delta_i(z_4,z_5))$. This $d_2(Z')$ is strictly above $d_2(X)$ in the lexicographic ordering of $\alpha \times \alpha$; this is because $\delta_i(z_2,z_4) = \min\{\delta_i(z_2,z_3),\delta_i(z_3,z_4)\} = \delta_i(z_2,z_3)$, and $\delta_i(z_4,z_5)$ must be strictly greater than $\delta_i(z_3,z_4)$, by the $i$-canonisation of $Z \setminus \{z_0,z_1\}$. It follows that $A'$ is as described.
    \end{claimproof}

    Write $\mathcal E \subseteq [\L]^\tau$ for the set of those $A \in [\L]^\tau$ satisfying the property described in Claim \ref{omegastarcase:claim}; our dense set $\D \subseteq [\L]^\tau$ in this case is the set of those $A \in \mathcal E$ satisfying the further condition that $\dmin_2(A)$ is minimal in $\{\dmin_2(A'): A' \in [A]^\tau\cap \mathcal E\}$, and the distinguished interval of such an $A$ is the leftmost almost $i$-canonised $X \in [A]^{\omega^*}_c$ with $d_2(X)$ minimal.\hfill $\dashv_\text{Case 1}$

    \begin{case}
        All of the infinite condensation classes in $\tau$ are ordered as $\omega$.
    \end{case}

    In this case, the definitions of both the dense set and the distinguished interval are less complex. First note the following:
    \begin{claim}\label{omegacase:claim}
        $\tau$ contains two consecutive condensation classes ordered as $\omega$. 
    \end{claim}

    \begin{claimproof}
        First note that since $\tau$ is scattered, every infinite interval of $\tau$ contains an infinite condensation class; an infinite order whose condensation classes are all finite is non-scattered, as e.g.\ the suborder consisting of the leftmost element of each condensation class is densely ordered. Since $\omega\omega^* \le \tau$ or $\omega^*\omega \le \tau$, it can certainly be split into multiple infinite intervals, each of which will contain at least one infinite condensation class.
        
        Fix some $A$ ordered as $\tau$ and let $X_0 < X_1$ be distinct condensation classes of $A$ ordered as $\omega$. If $X_0$ and $X_1$ are not consecutive, there is a non-empty interval of $A$ between them, which must be infinite, as otherwise it would be part of $X_1$; we therefore have that some condensation class between $X_0$ and $X_1$ is infinite, i.e.\ ordered as $\omega$. It follows that given any two condensation classes of $\tau$ ordered as $\omega$, either they are consecutive or there is another condensation class ordered as $\omega$ between them; if no two condensation classes ordered as $\omega$ are consecutive, it follows that these classes are densely ordered, contradicting the scatteredness of $\tau$.
        \end{claimproof}

        We now claim that any $A \in [\L]^\tau$ can be reduced to an $A' \in [A]^\tau$ with the following properties:

        \begin{enumerate}
            \item There is at least one $i$-canonised $X \in \ccomega(A')$ which is immediately followed by another element of $\ccomega(A')$;
            \item Let $\mathcal X(A')\subseteq \ccomega(A')$ denote the set of all $X$ as described in condition 1: then amongst those $X = \{x_n : n \in \omega\} \in \mathcal X(A')$ with $\delta_i(X) \coloneqq \delta_i(x_0,x_1)$ minimal, there is a leftmost such $X$.
        \end{enumerate}
        Condition 1 is easy to satisfy, as given any $X \in \ccomega(A)$ immediately followed by another element of $\ccomega(A)$ we can simply reduce $X$ to some $X' \in [X]^\omega$ which is $i$-canonised. For condition 2, suppose there are multiple $Z \in \mathcal X(A)$ all taking the same value of $\delta_i(Z)$. Pick one, $X$, say, and reduce to $A' \in [A]^\tau$ by, for each of these $Z < X$, removing $z_0$; then for any such $Z$, $Z \setminus \{z_0\} \in \ccomega(A')$ has $\delta_i(Z') = \delta_i(z_1,z_2)$, which is strictly greater than $\delta_i(z_0,z_1) = \delta_i(Z)$ as $Z$ is $i$-canonised by assumption.

        Let us write $\mathcal E \subseteq [\L]^\tau$ for the set of all $A \in [\L]^\tau$ satisfying conditions 1 and 2 above, and write $\deltaimin(A) \coloneqq \min\{\delta_i(X) : X \in \mathcal X(A)\}$. Our dense set $\D$ in this case is the set of all those $A \in \mathcal E$ with the additional property that 
        \[
        \deltaimin (A) = \min\{\deltaimin(A') : A' \in [A]^\tau \cap \mathcal E\};
        \] the distinguished interval of such an $A$ is the condensation class $X$ described in condition 2. \hfill $\dashv_\text{Case 2}$

        We are now ready to define our colouring $c : \D \to 2$. For $A \in \D$, let us write $\di(A)$ for the distinguished interval of $A$, and write $A^+ \coloneqq \{x \in A: x > \di(A)\}$ for the final segment of $A$ above its distinguished interval. In each case, $\di(A)$ is immediately followed by (at least) two consecutive elements of $A$; write these as $x_A < y_A$. Now, if $A \in \D$ is such that $i$ is neither order-preserving nor order-reversing on $A^+$, set
        \[c(A) \coloneqq \begin{cases*}
            0 & if $i(x_A) \lex i(y_A)$;\\
            1 & if $i(x_A) >_{\text{lex}} i(y_A)$.
        \end{cases*}\]

        If instead $i$ is order-preserving or order-reversing on $A^+$, set
        \[c(A) = C(i \im A^+),\]
        where $C$ is the colouring defined in \cite[Theorem 2]{higheranalogues}, which simultaneously witnesses that $\twoalphalex \centernot \rightarrow (\varphi)^\varphi$ for all well-orderable scattered $\varphi$ with $\omega\omega^* \le \varphi$ or $\omega^*\omega \le \varphi$.

        \begin{claim}\label{pickoutanyxy:claim}
            Let $A \in \D$ and let $x< y \in A^+$. Then there exists $A' \in [A]^\tau \cap \D$ with $x = x_{A'}$, $y = y_{A'}$.
        \end{claim}
        \begin{claimproof}
            Let us write $\tau^+$ for the order type of $A^+$; we find our required $A'$ by reducing $A^+$ to some $B \in [A^+]^{\tau^+}$ and leaving $A \setminus A^+$ untouched. Observe that any $A' \in [A]^\tau$ of this form is necessarily still in $\D$, and has the same distinguished interval as $A$: In case 1, $B$ may contain new almost $i$-canonised intervals ordered as $\omega^*$, $Z$, say; by our assumption that $\dmin_2 (A)$ was already minimal, however, it must be the case that $d_2(Z) \ge d_2(\di(A))$ for any such $Z$, and so $\di(A)$ is still the leftmost almost $i$-canonised element of $[A']^{\omega^*}_c$ with $d_2(\di(A))$ minimal; it follows that $A' \in \D$ and $\di(A') = \di(A)$. Similarly, in case 2, any $Z \in \mathcal X(B)$ must have $\delta_i(Z) \ge \delta_i(\di(A))$, and so $\di(A') = \di(A)$.
            
            We find our desired $B$ like so. Since $\tau$ is strictly indecomposable to the right, $A^+ \cap (y,\rightarrow)$ contains isomorphic copies of any final segment of $A$; let $B \in [A^+]^{\tau^+}$ be of the form $B = \{x,y\} \cup B'$, where $B' \subseteq A^+ \cap (y,\rightarrow)$. Then set $A' \coloneqq (A\setminus A^+) \cup B$. Now $\di(A') = \di(A)$, and since $\di(A)$ is immediately followed by $x$ and $y$ in $A'$, we have that $x_{A'} = x$ and $y_{A'} = y$.
        \end{claimproof}
        We now show that no $A \in \D$ can be homogeneous for $c$. Let $A \in \D$ be arbitrary. If for some $A' \in [A]^\tau \cap \D$, $i$ is either order-preserving or order-reversing on $A'^+$, reduce to $A'$; then since $i \im A'^+$ is not homogeneous for $C$, we can find some $B \subseteq i \im A'^+$ with the same order type under $\lex$ but with $C(B) \neq C(i \im A'^+)$; then $A'' \coloneqq (A \setminus A')\cup i^{-1}(B)$ has $c(A'') = C(B) \neq c(A')$, and so $A$ is not homogeneous for $c$.
        
        So suppose instead that no such $A'$ exists. In particular, since $i$ is neither order-preserving nor order-reversing on $A^+$, we can find $x_0 < y_0$, $x_1 < y_1$ in $A^+$ such that $i(x_0) <_\text{lex} i(y_0)$ and $i(x_1) >_\text{lex} i(y_1)$. Then by Claim \ref{pickoutanyxy:claim}, we can find $A_0, A_1 \in [A]^\tau \cap \D$ with $x_{A_0} = x_0$, $y_{A_0} = y_0$, $x_{A_1} = x_1$, and $y_{A_1} = y_1$; then $c(A_0) = 0$ and $c(A_1) = 1$, and so once again $A$ is not homogeneous for $c$.
\end{proof}
\begin{proof}[Proof of Theorem \ref{omegaomegastarkw:thm}]
\setcounter{scratch}{\value{thmcount}}
\setcounter{thmcount}{\getrefnumber{omegaomegastarkw:thm}}
\setcounter{claimcount}{0}
    Let $\tau$ be a well-orderable scattered order type with $\omega\omega^* \le \tau$ or $\omega^*\omega \le \tau$. We appeal to the following result of Laver:
    \begin{subfact}\label{sumofindecs:fact} \textnormal{(\!\!\cite{laver})}
        Any scattered linear order can be written as a finite sum of indecomposable scattered linear orders.\footnote{See the footnote after Fact \ref{jullien:fact}.}
    \end{subfact}
    We can therefore decompose $\tau$ as a sum $\tau = \tau_0 + \tau_1 + \dots + \tau_n$, where each $\tau_k$ is indecomposable; fix some such decomposition. Given a copy of $\tau$, there are potentially many different ways of decomposing it as an interval ordered as $\tau_0$ followed by an interval ordered as $\tau_1$, etc. Whichever of $\omega\omega^*$ or $\omega^*\omega$ embeds in $\tau$ must embed in at least one $\tau_k$; fix some $m$ such that $\omega\omega^* \le \tau_m$ or $\omega^*\omega \le \tau_m$. We define our colouring on a dense subset $\D \subseteq [\L]^\tau$ such that for any $A \in \D$, it is not necessarily possible to pick out an interval of $A$ ordered as $\tau_m$, but it is possible to pick out a final segment of an interval of $A$ ordered as $\tau_m$.
    
    Since $\tau_m$ is scattered and indecomposable, it is either strictly indecomposable to the right or to the left, and since it is scattered and infinite, it must have infinite condensation classes. Wlog $\tau_m$ is strictly indecomposable to the right. For $A \in [\L]^\tau$, $X \in [A]^\omega_c \cup [A]^{\omega^*}_c$, we associate each $X$ with the quantity $d_0(X) = (\delta_i(x_0,x_1),\delta_i(x_1,x_2))$. % $\delta(X) = \langle\delta_i(x_0,x_1),\delta_i(x_1,x_2)\rangle$.
    Now define $\mathcal E \subseteq [\L]^\tau$ to be the set of all $A \in [\L]^\tau$ with the following properties:
    \begin{enumerate}
        \item There is some $i$-canonised $X \in [A]^\omega_c \cup [A]^{\omega^*}_c$ such that for some decomposition $A = A_0 \cup A_1 \cup \dots \cup A_n$ of $A$ into consecutive intervals ordered as $\tau_0, \tau_1, \dots, \tau_n$, respectively, $X \in A_m$;
        \item Amongst all $X$ as described in condition 1 with $d_0(X)$ minimal, there is a leftmost such $X$.
    \end{enumerate}
    Observe that $\mathcal E$ is dense: given any $A \in [\L]^\tau$, we may decompose it into pieces $A_0 \cup A_1 \cup \dots \cup A_n$ and then $i$-canonise some $\omega$- or $\omega^*$-interval of $A_m$, yielding condition 1; given multiple such $i$-canonised $Z \in [A]^\omega_c \cup [A]^{\omega^*}_c$ with $d_0(Z)$ minimal, we may fix some $X$ and remove $z_1$ for each such $Z < X$ to attain condition 2.

    Now, for any $A \in \mathcal E$, let $X = X(A)$ be the interval guaranteed by condition 2; then there is \emph{some} decomposition of $A$ into consecutive intervals ordered as $\tau_k$ such that $X$ is in the piece ordered as $\tau_m$. We cannot in general determine this decomposition from $X$ alone, nor even the interval ordered as $\tau_m$, but since $\tau_m$ is strictly indecomposable to the right, $X$ is immediately followed in $A$ by a final segment of a copy of $\tau_m$ (in particular, $X$ is not cofinal in $\tau_m$, as $\omega < \tau_m$ and $\tau_m$ is strictly indecomposable to the right), and this final segment is uniquely determined; it is the set $\operatorname{fi}(A) \coloneqq \{a \in A: a > X \text{ and }\tau_m \not \le [x_0, a]_A\}$.

    We therefore have a canonical way of picking out an interval of $A$ ordered as some type equimorphic with $\tau_m$, $\operatorname{fi}(A)$; we would now like to simply apply the colouring from the proof of Lemma \ref{indecscatteredkw:lemma} to $\operatorname{fi}(A)$. In order for this to work, we need to reduce further to another dense subset of $\mathcal E$: for $A \in \mathcal E$, write $\dmin_0(A) \coloneqq d_0(X(A))$, and let $\D \subseteq \mathcal E$ be the set of those $A \in \mathcal E$ such that $\dmin_0(A)$ is minimal amongst $\{\dmin_0(A') : A' \in [A]^\tau \cap \mathcal E\}$. Then for any $A \in \D$ and any $B \in [\operatorname{fi}(A)]^{\otp \operatorname{fi}(A)}$, $A' \coloneqq (A \setminus \operatorname{fi}(A))\cup B \in [A]^\tau \cap \D$ and $\operatorname{fi}(A') = B$. It follows that the colouring $c' : \D \to 2$ given by, for $A \in \D$,
    \[c'(A) = c(\operatorname{fi}(A)),\]
    where $c$ is the colouring defined in the proof of Lemma \ref{indecscatteredkw:lemma}, has no homogeneous set.
\end{proof}\setcounter{thmcount}{\value{scratch}}
Our final result in this section shows that the graph-theoretic IEPR shown consistent in Proposition \ref{graphconsistency:prop} yields a failure not just of \s{KWP}$_1$ but also of \s{O}.
\begin{prop}\label{randomgraph:prop}
    Let $G$ be a graph which is linearly orderable, i.e.\ whose vertex set is linearly orderable. Then
    \[G \centernot \rightarrow (R)^R,\]
    where $R$ is the random graph.
\end{prop}
\begin{proof}
    Our proof uses some ideas from \cite[Theorem 2]{pigeonholeorders}, adapted to the choiceless setting. The key property of the random graph for the purposes of this proof is the following.
    \begin{subfact}
        The random graph $R$ has the \emph{pigeonhole property}, i.e.\ whenever (the vertex set of) $R$ is partitioned into finitely many disjoint pieces $R = R_0 \cup R_1 \cup \dots \cup R_n$, there is some $m$ such that $R_m \cong R$.
    \end{subfact}
     See e.g.\ \cite[Proposition 3]{cameronrandomgraph} for a proof; Choice is not used. Let $G$ be a graph with $[G]^R \neq \emptyset$, and let $<$ be a linear ordering on $G$. We build a colouring of $ c: [G]^R \to 2$ with no homogeneous set in the following way. The colour of $A \in [G]^R$ will be determined by its order-theoretic properties under the ordering $<$, and those of its subcopies of $R$. If $A$ is such that there are $A' \in [A]^R$ with a $<$-minimal element and $A'' \in [A]^R$ with no $<$-minimal element, set
     \[c(A) \coloneqq \begin{cases*}
         0 & if $A$ has a minimal element;\\
         1 & if $A$ has no minimal element.
     \end{cases*}\]
     If this is not the case, but $A$ is such that some $A' \in [A]^R$ has a maximal element and some $A'' \in [A]^R$ has no maximal element, set
     \[c(A) \coloneqq \begin{cases*}
         0 & if $A$ has a maximal element;\\
         1 & if $A$ has no maximal element.
     \end{cases*}\]
     It is clear from just this partial definition of $c$ that for $A \in [G]^R$ to be homogeneous for $c$, we can assume (by reducing if necessary) that either every $A' \in [A]^R$ has a minimal element or no $A' \in [A]^R$ has a minimal element, and similarly for maximal elements.

     We show first that in order for $A$ to be homogeneous, it cannot be the case that every $A' \in [A]^R$ has neither a minimal element nor a maximal element. Let $a < b \in A$; then define $A_0, A_1 \subseteq A$ by
     \begin{align*}
         A_0 &\coloneqq \{a\} \cup \{x \in A : b < x\};\\
         A_1 & \coloneqq A \setminus A_0.
     \end{align*}
     Then $A_0$ has minimal element $a$ and $A_1$ has maximal element $b$. By the pigeonhole property, at least one of $A_0$, $A_1$ is isomorphic to $R$, and so we obtain some $A_i \in [A]^R$ which has either a minimal element or a maximal element.

     If $A$ is homogeneous it therefore must either be the case that every element of $[A]^R$ has a minimal element or that every element of $[A]^R$ has a maximal element. 
     \begin{claim}
         Let $A$ be such that every element of $[A]^R$ has a minimal element. Then there exists $S \in [A]^R$ which is well-ordered by $<$.
     \end{claim}
     \begin{claimproof}
         Let $S$ be the maximal well-ordered $<$-initial segment of $A$; we claim that $S \in [A]^R$. To see this, observe that $A \setminus S$ has no minimal element; if $a = \min A \setminus S$, then $(\leftarrow, a] \cap A = S \cup \{a\}$ is well-ordered by $<$, as $S$ is well-ordered by $<$ and $a > S$, so $a \in S$, a contradiction. By the pigeonhole principle, one of $S$ or $A\setminus S$ is in $[A]^R$; since all elements of $[A]^R$ have a minimal element, this must be $S$.
     \end{claimproof}
     Similarly, if $A$ is such that every element of $[A]^R$ has a maximal element, there is some $S \in [A]^R$ which is anti-well-ordered by $<$.
     
     We now complete the definition of our colouring. If $A \in [G]^R$ is well-ordered by $<$, set
     \[c(A) \coloneqq \begin{cases*}
         0 & if $a_0Ea_1$;\\
         1 & if $\lnot a_0Ea_1$,
     \end{cases*}\]
     where $a_0$, $a_1$ are the first two elements of $A$ under $<$; similarly, if $A$ is anti-well-ordered by $<$, set
     \[c(A) \coloneqq \begin{cases*}
         0 & if $a_0Ea_1$;\\
         1 & if $\lnot a_0Ea_1$,
     \end{cases*}\]
     where $a_0$, $a_1$ are the last two elements of $A$ under $<$. In all remaining undefined cases, the colour of $A$ does not matter; set $c(A)=0$.

     We claim that this colouring has no homogeneous set. Let $A \in [G]^R$. For $A$ to be homogeneous for $c$, it must be the case either that every element of $[A]^R$ has a minimal element or that every element of $[A]^R$ has a maximal element; wlog every element of $[A]^R$ has a minimal element. Then we may assume, by reducing further if necessary, that $A$ is well-ordered by $<$, and that $\otp\langle A,<\rangle$ is minimal in $\{\otp \langle A',<\rangle : A' \in [A]^R\}$. Then for any $B \subseteq A$ which has strictly smaller order type than $A$, it must be the case that $A \setminus B \in [A]^R$. In particular, letting $x, y \in A$, $x < y$ be such that
     \[xEy \iff \lnot a_0 E a_1,\]
     we see that $A' \coloneqq \{x,y\} \cup \{a \in A : a > y\}$ is an element of $[A]^R$; but since $a'_0 = x$ and $a'_1 = y$, we have $c(A) \neq c(A')$, and so $A$ is not homogeneous for $c$.
     \end{proof}
     We remark that exactly the same proof gives a more general result. Let us say that a relation $r$ of arity $n$ on a structure $A$ is \emph{trivialised} by the linear order $\Aa$ if whenever $a_0 < a_1 < \dots < a_{n-1}$ and $a'_0 < a'_1 < \dots < a'_{n-1}$, 
     \[A \models r(a_0,a_1,\dots,a_{n-1}) \iff r(a_0',a_1',\dots,a_{n-1}').\]
     The final step of the proof of Proposition \ref{randomgraph:prop} simply relies on the fact that the edge relation $E$ is not trivialised by any linear ordering of $R$. As such, we obtain the following generalisation:
     \begin{thm}\label{pigeonholegeneral:thm}
         Let $\mathcal L$ be a relational language, let $A$ be an $\mathcal L$-structure whose underlying set is linearly orderable and let $B$ be an $\mathcal L$-structure with the pigeonhole property with at least one relation which is not trivialised by any linear order on $B$. %Assume either \s{DC} or that $B$ is well-orderable.
         Then
        \[A \centernot \rightarrow (B)^B.\]
     \end{thm}
\subsection*{Acknowledgements}
    This research was funded in whole or in part by the Austrian Science Fund (FWF) [10.55776/ESP5711024]. For open access purposes, the authors have applied a CC BY public copyright license to any author-accepted manuscript version arising from this submission.

\end{document}

%% file: header.tex
\usepackage{graphicx} % Required for inserting images
\usepackage{amsmath}
\usepackage{amstext}
\usepackage{mathtools}
\usepackage{amssymb}
\usepackage{textcomp}
\usepackage{listings}
\usepackage{graphicx}
\usepackage{physics}
\usepackage{relsize}
\usepackage{exscale}
\usepackage{amsthm}
\usepackage{amsthm,amsmath}
\usepackage{amsmath,mathdots}
\usepackage{bbm}
\usepackage{mathrsfs}
\usepackage{tikz-cd}
\usepackage{colonequals}
\usepackage{centernot}
\usepackage{thmtools, thm-restate}
\usepackage{refcount}
\newcommand{\s}[1]{\textup{\textsf{#1}}}

\renewcommand{\L}{\langle L,<\rangle}

\newcommand{\Q}{\mathbb{Q}}

\newcommand{\im}{\mathbin{\hbox{\tt\char'42}}}

\DeclareMathOperator{\sym}{sym}
\DeclareMathOperator{\fix}{fix}

\newcommand{\twoalphalex}{\langle {}^\alpha 2,<_{\operatorname{lex}}\rangle}
\newcommand{\twoalpha}{{}^\alpha 2}

\newcommand{\lessomega}{{}^{<\omega}2}
\newcommand{\lessomegatree}{\langle \lessomega, \sqsubseteq \rangle}
\newcommand{\lessomegalex}{\langle \lessomega, \lex \rangle}

\newcommand{\lex}{<_{\operatorname{lex}}}

\newcommand{\ccomega}{\mathbf{cc}^\omega}
\newcommand{\ccomegastar}{\mathbf{cc}^{\omega^*}}

\DeclareMathOperator{\otp}{otp}
\newcommand{\Aa}{\langle A, < \rangle}

\newcommand{\Aprec}{\langle A,\prec \rangle}
\newcommand{\Bprec}{\langle B,\prec \rangle}
\newcommand{\Ii}{\langle I, < \rangle}
\newcommand{\Iprec}{\langle I, \prec \rangle}
\newcommand{\Jj}{\langle J, < \rangle}
\newcommand{\Jprec}{\langle J, \prec \rangle}
\newcommand{\Qq}{\langle \Q, < \rangle}
\newcommand{\Qprec}{\langle \Q, \prec \rangle}

\newcommand{\szero}{s ^\frown \langle 0 \rangle}
\newcommand{\sone}{s ^\frown \langle 1 \rangle}

\newcommand{\D}{\mathscr D}
\newcounter{thmcount}
\newtheorem{thm}[thmcount]{Theorem}
\newtheorem{prop}[thmcount]{Proposition}
\newtheorem{defn}[thmcount]{Definition}
\newtheorem{lemma}[thmcount]{Lemma}

\newtheorem{obs}[thmcount]{Observation}
\newtheorem*{thm*}{Theorem}

\newtheorem{fact}[thmcount]{Fact}
\theoremstyle{definition}

\newcounter{claimcount}
\numberwithin{claimcount}{thmcount}
\newtheorem{claim}[claimcount]{Claim}
\newcommand*{\claimproofname}{Proof of claim}
\newenvironment{claimproof}[1][\claimproofname]{\begin{proof}[#1]}{\end{proof}}
\newtheorem{subfact}[claimcount]{Fact}

\counterwithin{thmcount2}{section}

\newcounter{casecount}
\newtheorem{case}[casecount]{Case}
\newcommand{\dmin}{d^\textnormal{min}}
\newcommand{\deltaimin}{\delta_i^\textnormal{min}}

\usepackage{cite}
\usepackage[shortlabels]{enumitem}
\usepackage[new]{old-arrows}
\usepackage{slashed}

\newcounter{scratch}

\usepackage{lipsum}
\newcommand\blfootnote[1]{%
  \begingroup
  \renewcommand\thefootnote{}\footnote{#1}%
  \addtocounter{footnote}{-1}%
  \endgroup
}